\documentclass[12pt]{article}
\input epsf.tex

\usepackage{amsmath}
\usepackage{amsthm}
\usepackage{amsfonts}
\usepackage{amssymb}
\usepackage{graphicx}
\usepackage{latexsym}

\usepackage{amsmath,amsthm,amsfonts,amssymb}

\usepackage{epsfig}

\usepackage{amssymb}
\usepackage[all]{xy}
\xyoption{poly}
\usepackage{fancyhdr}
\usepackage{wrapfig}

\theoremstyle{plain}
\newtheorem{thm}{Theorem}[section]
\newtheorem{prop}[thm]{Proposition}
\newtheorem{lem}[thm]{Lemma}
\newtheorem{cor}[thm]{Corollary}

\theoremstyle{definition}

\theoremstyle{remark}
\newtheorem{remark}{Remark}

\newtheorem{problem}{Problem}

\advance \topmargin by -\headheight
\advance \topmargin by -\headsep
\evensidemargin \oddsidemargin
\def\Aut{\textnormal{Aut}}

\def\cc{{\curvearrowright}}

\def\CInd{\textnormal{CInd}}

\def\E{{\mathbb E}}

\def\F{{\mathbb F}}

\def\cH{{\mathcal H}}

\def\mc{\mathcal}
\def\N{{\mathbb N}}

\def\R{{\mathbb R}}

\def\chix{{\raise.5ex\hbox{$\chi$}}}

\def\Z{{\mathbb Z}}

\newcommand{\ra}{\rightarrow}
\newcommand{\csuchthat}{ :~ }

\begin{document}
\title{On a co-induction question of Kechris }
\author{Lewis Bowen\footnote{email:lpbowen@math.tamu.edu} \\ Texas A\&M University \\ $\ $ \\ Robin D. Tucker-Drob\footnote{email:rtuckerd@caltech.edu} \\ Caltech}
\begin{abstract}
This note answers a question of Kechris:  if $H<G$ is a normal subgroup of a countable group $G$, $H$ has property MD and $G/H$ is amenable and residually finite then $G$ also has property MD. Under the same hypothesis we prove that for any action $a$ of $G$, if $b$ is a free action of $G/H$, and $b_G$ is the induced action of $G$ then $\CInd_H^G(a|H) \times b_G$ weakly contains $a$.  Moreover,  if $H<G$ is any subgroup of a countable group $G$, and the action of $G$ on $G/H$ is amenable, then $\CInd_H^G(a|H)$ weakly contains $a$ whenever $a$ is a Gaussian action.
\end{abstract}
\maketitle
\noindent
{\bf Keywords}: co-induction, weak containment, MD, Gaussian action\\
{\bf MSC}:37A35\\

\noindent

\section{Introduction}\label{sec:intro}

The Rohlin Lemma plays a prominent role in classical ergodic theory. Roughly speaking, it states that any aperiodic automorphism $T$ of a standard non-atomic probability space $(X,\mu)$ can be approximated by periodic automorphisms. In \cite{OW80}, Ornstein and Weiss generalized the Rohlin Lemma to actions of amenable groups and used it to extend many classical ergodic theory results (such as Ornstein theory) to the amenable setting. 

There is no analogue of the Rohlin Lemma for non-amenable groups. However, one can hope to understand more precisely how and why this is so. The concept of ``weak containment'' of actions, introduced by A. Kechris \cite{Ke10}, is a natural starting point. To be precise, let $(X,\mu)$, $(Y,\nu)$ be standard non-atomic probability spaces. Let $G \cc^a (X,\mu)$, $G \cc^b (Y,\nu)$ be probability-measure-preserving (p.m.p.) actions. An {\em observable} $\phi$ for $a$ is a measurable map $\phi:X \to \N$. For $F \subset G$, let $\phi_a^F:X \to \N^F=\{y:F \to \N \}$ be defined by
$$\phi_a^F(x)(f)=\phi(a(f)x).$$
Then $a$ is said to be {\em weakly contained} in $b$ (denoted $a \prec b$) if for every $\epsilon>0$, every finite $F \subset G$, every observable $\phi$ for $a$, there is an observable $\psi$ for $b$ such that
$$\| \phi^F_*\mu - \psi^F_*\nu\|_1 \le \epsilon.$$
The two actions are {\em weakly equivalent} if $a \prec b$ and $b \prec a$. 

If $G$ is infinite and amenable, then as remarked in \cite{Ke11}, if $a$ is a free action then $a$ weakly contains every action of $G$. This is essentially equivalent to the Rohlin Lemma for amenable groups. However, when $G$ is non-amenable then it may possess uncountably many free non-weakly equivalent actions \cite{AE11}. It is unknown whether the same holds true for every non-amenable group.

It is natural to ask how weak equivalence behaves with respect to operations such as co-induction. To be precise, let $H<G$ be a subgroup. Let $H \cc^a (X,\mu)$ be a p.m.p. action. Let $Z=\{ z \in X^G:~ a(h^{-1})z(g) = z(gh)~\forall h\in H, g\in G\}$. Let $G \cc^b Z$ be the action $(b(g)z)(f)=z(g^{-1}f)$ for $g,f\in G, z \in Z$. 

A {\em section} of $H$ in $G$ is a map $\sigma:G/H \to G$ such that $\sigma(gH) \subset gH$ for every $g\in G$. Let us assume $\sigma(H)=e$. Define $\Phi:Z \to X^{G/H}$ by $\phi(z)(gH) = z(\sigma(gH))$. This is a bijection. Define a measure $\zeta$ on $Z$ by pulling back the product measure $\mu^{G/H}$ on $X^{G/H}$. Then $G \cc^b (Z,\zeta)$ is probability-measure-preserving. This action is called said to be {\em co-induced} from $a$ and is denoted $b=\CInd_H^G(a)$.

Problem A.4. of \cite{Ke11} asks the following.
\begin{problem}
Let $G$ be a countable group with a subgroup $H<G$. Suppose the action of $G$ on $G/H$ is amenable. Is it true that for any p.m.p. action $a$ of $G$ on a standard non-atomic probability space, the co-induced action $\CInd_H^G(a|H)$ weakly contains $a$?
\end{problem}
A positive answer can be interpreted as providing a relative version of the Rohlin lemma. Note that the action of $G$ on $G/H$ being amenable is a necessary condition, since if we take $a$ to be the trivial action $\tau _G$ of $G$ on a standard non-atomic probability space $(X,\mu )$, then $\mbox{CInd}_H^G (\tau _G|H)$ is isomorphic to the generalized Bernoulli shift action $s_{G,G/H ,X}$ of $G$ on $X^{G/H}$ (see section \ref{sec:conseq}), and $s_{G,G/H,X}$ weakly containing $\tau _G$ is equivalent to the action of $G$ on $G/H$ being amenable by \cite{KT08}. Also note that if replace the actions with unitary representations, then the analogous problem is known to have a positive answer (this is E.2.6 of \cite{BdlHV08}).

Our main results solve Problem 1 in a number of cases and provide applications to property MD. To begin, we prove:
\begin{thm}\label{thm:main} 
Let $G$ be a countable group with normal subgroup $H$. Suppose that $G/H$ is amenable and that $|G/H|=\infty$. Let $b$ be any free p.m.p. action of $G/H$. Let $b_G$ be the associated action of $G$ (i.e., $b_G$ is obtained by pre-composing $b$ with the quotient map $G \to G/H$). Then for any p.m.p. action $a$ of $G$ on standard non-atomic probability space, the product action $\CInd_H^G(a|H) \times b_G$ weakly contains $a$.

\end{thm}

Taking $b$ to be the Bernoulli shift action of $G/H$ over a standard non-atomic probability base space, we show that Theorem \ref{thm:main} implies (see \ref{cor:cind} below)
\[
a\prec \mbox{CInd}_H^G ((a\times \tau _G)|H)
\]
where $\tau _G$ is the trivial action of $G$ as above.  In particular, if $a|H$ weakly contains $(a\times \tau _G)|H$, then $\CInd_H^G(a|H)$ weakly contains $a$.  For instance, by \cite{AW11} this is the case whenever $a$ is an ergodic p.m.p. action of $G$ that is not strongly ergodic.  This also holds when $a$ is a \emph{universal} action of $G$, i.e., $b\prec a$ for every p.m.p. action $b$ of $G$. That such actions exist for every countable group $G$ is due to Glasner-Thouvenot-Weiss \cite{GTW06} and, independently, to Hjorth (unpublished, see 10.7 of \cite{Ke10}). This has the following consequence:

\begin{thm}\label{thm:univ}
Let $G$ and $H$ be as in Theorem \ref{thm:main}. If $b$ is a universal action of $H$ then $\CInd_H^G(b)$ is a universal action of $G$.
\end{thm}

In section \ref{sec:gactions} we describe the Gaussian action construction.  For every real positive definite function $\varphi$ defined on a countable set $T$, a probability measure $\mu _\varphi$ on $\R ^T$ is defined, and we call $(\R ^T ,\mu _\varphi )$ a \emph{Gaussian probability space}.  When $G$ acts on $T$ and $\varphi$ is invariant for this action, then $\mu _\varphi$ will be an invariant measure for the shift action of $G$ on $(\R ^T ,\mu _\varphi )$.  A p.m.p. action $a$ of $G$ is called a \emph{Gaussian action} if it is isomorphic to the shift action of $G$ on some Gaussian probability space $(\R ^T , \mu _\varphi )$ associated to an invariant positive definite function $\varphi$.  We show that Problem 1 always has a positive answer for Gaussian actions.

\begin{thm}\label{thm:gaussian}
Let $G$ be a countable group with a subgroup $H<G$. Suppose the action of $G$ on $G/H$ is amenable. Then the co-induced action $\CInd_H^G(a|H)$ weakly contains $a$ for every Gaussian action $a$ of $G$.
\end{thm}

Part of the motivation for posing Problem 1 above concerns a property of groups introduced by Kechris called property MD.  To be precise, let $G$ be a residually finite group, and let $\rho_G$ be the canonical action of $G$ on its profinite completion. Recall that $\tau_G$ is the trivial action of $G$ on $(X,\mu)$, a standard non-atomic probability space.  Then $G$ has MD if and only if every p.m.p. action of $G$ is weakly contained in the product action $\tau_G \times \rho_G$. 

The property MD is an ergodic theoretic analog of the property FD discussed
in Lubotzky-Shalom \cite{LS04} (see also Lubotzky-Zuk \cite{LZ03}). This asserts
that the finite unitary representations of $G$ on an infinite-dimensional separable
Hilbert space $\cH$ are dense in the space of unitary representations
of $G$ in $\cH$. It is not difficult to show that $MD \Rightarrow FD$ but the converse is unknown. 

It is known (see \cite{Ke11} for more details), that the following groups have MD: residually finite amenable groups, free products of finite groups, subgroups of MD groups, finite extensions of MD groups. On the other hand, various groups such as $SL_n(\Z)$ for $n>2$ are known not to have FD \cite{LS04} \cite{LZ03} and hence also do not have MD. It is an open question whether the direct product of two free groups has MD.

In \cite{Ke11}, Conjecture 4.14, Kechris conjectured the following:
\begin{thm}\label{conj:main}
Let $N$ be an infinite, residually finite 
group satisfying MD. Let $N \vartriangleleft G$ with $G$ residually finite. Assume that: 
\begin{enumerate}
\item For every $H \vartriangleleft N$ with $[N : H ] < \infty$, there is $G' \vartriangleleft G$ such that $G' \subset H$ 
and $[N : G' ] < \infty$. 
\item $G/N$ is a residually finite, amenable group. 
\end{enumerate}
Then $G$ satisfies MD. 
\end{thm}
As noted in \cite{Ke11}, this result implies that surface groups and the fundamental groups of virtually fibered closed hyperbolic 3-manifolds, (e.g., $SL_2(\Z[i])$) have property MD. This follows from the fact that free groups have property MD (proven in \cite{Ke11} and in different terminology in \cite{Bo03}). Kechris proved that an affirmative answer to Problem 1 above implies Theorem \ref{conj:main}. Our proof follows his line of argument.

Note: If $N$ is finitely generated then the first condition of Theorem \ref{conj:main} is automatically satisfied since if $N$ is normal in $G$ and $H<N$ has finite index, then for every $g\in G$, $gHg^{-1}$ is a subgroup of $N$ with the same index as $H$. Because $N$ is finitely generated, this implies there are only finitely many different conjugates of $H$. The intersection of all these conjugates is a normal subgroup in $G$ with finite-index in $N$. 


{\bf Acknowledgements}: We would like to thank Alekos Kechris for encouraging us to take on this problem and for many valuable comments on an earlier draft of this paper. L.B. was partially supported by NSF grants DMS-0968762 and DMS-0954606.

\section{The space of actions and proof of Theorem \ref{conj:main}}

Let $(X,\mu)$ denote a standard non-atomic probability space and $A(G,X,\mu)$ the set of all p.m.p. actions of $G$ on $(X,\mu)$. This set is naturally identified with a subset of the product space $\Aut(X,\mu)^G$ where $\Aut(X,\mu)$ denotes the space of all automorphisms of $(X,\mu)$. We equip the $\Aut(X,\mu)$ with the weak topology, $\Aut(X,\mu)^G$ with the product topology, and $A(G,X,\mu)$ with the subspace topology (also called the weak topology). The group $\Aut(X,\mu)$ acts on $A(G,X,\mu)$ by $(T a)(g) = T a(g)T^{-1}$ for all $T \in \Aut(X,\mu)$, $a\in A(G,X,\mu)$ and $g\in G$. The orbit of $a$ under this action is called its {\em conjugacy class}.

\begin{lem}\label{lem:preclem}
Let $a,b \in A(G,X,\mu)$. Then $a \prec b$ if and only if $a$ is contained in the (weak) closure of the conjugacy class of $b$.
\end{lem}
\begin{proof}
This is Proposition 10.1 of \cite{Ke10}. 
\end{proof}

An action $a \in A(G,X,\mu)$ is {\em finite} if it factors through the action of a finite group. From lemma \ref{lem:preclem} it follows that for any $a\in A(G,X,\mu)$, $a \prec \tau_G \times \rho_G$ if and only if $a$ is contained in the (weak) closure of the set of finite actions (this is implied by the proof of Proposition 4.8 \cite{Ke11}).

We need the following lemmas.
\begin{lem}
Let $a, b$ be actions of a countable group $G$. If $a$ and $b$ are weakly contained in $\tau_G \times \rho_G$ then $a \times b$ is weakly contained in $\tau_G \times \rho_G$. \end{lem}

\begin{proof}
If $a$ is a weak limit of finite actions $a_i$ and $b$ is a weak limit of finite actions $b_i$ then $a\times b$ is the weak limit of $a_i\times b_i$. 
\end{proof}

\begin{lem}
 If $H<G$ is a normal subgroup, $G/H$ is amenable and residually finite, and $b$ is a p.m.p. action of $G/H$ then the induced action $b_G$ of $G$ is weakly contained in $\tau_G \times \rho_G$.
\end{lem}

\begin{proof}
As noted in \cite{Ke11}, because $G/H$ is residually finite and amenable, it has MD. Therefore, $b$ is a weak limit of finite actions $b_i$ of $G/H$. If $b_{G,i}$ are the induced actions of $G$, then the $b_{G,i}$ are also finite and $b_{G,i}$ converges weakly to $b_G$.
\end{proof}

\begin{proof}[Proof of Theorem \ref{conj:main} from Theorem \ref{thm:main}]
Let $a$ be a p.m.p. action of $G$. In \cite{Ke11} section 4, it is shown that $\CInd_N^G(a|N)$ is weakly contained in $\tau_G \times \rho_G$. Let $b$ be a free p.m.p. action of $G/N$. Because $G/N$ is amenable the previous lemmas imply $\CInd_N^G(a|N)\times b_G \prec \tau_G \times \rho_G$. So Theorem \ref{thm:main} implies $a \prec \CInd_N^G(a|N)\times b_G \prec \tau_G \times \rho_G$. Since $a$ is arbitrary, $G$ has MD.
\end{proof}

\section{The Rohlin Lemma}

The purpose of this section is to prove:
\begin{thm}\label{thm:tiling}
If $G$ is a countably infinite amenable group then for every free p.m.p. action $G \cc^a (X,\mu)$, every finite $F \subset G$ and $\epsilon>0$ there is a measurable map $J: X \to G$ such that
$$\mu(\{ x\in X:~ J(a(f)x)=fJ(x) ~\forall f\in F\}) \ge 1-\epsilon.$$
\end{thm}

This will follow easily from the following version of the Rohlin Lemma due to Ollagnier \cite{Ol85} Corollary 8.3.12 (see 2.2.8. for the definition of $M(D,\delta)$).

\begin{thm}
Let $G \cc  (X,\mu)$ be as above. Then for every finite $F \subset G$, for every $\delta,\eta>0$ there exists a finite collection $\{(\Lambda_i,A_i)\}_{i\in I}$ satisfying:
\begin{enumerate}
\item for every $i\in I$, $\Lambda_i \subset G$ is finite and 
$$ \frac{ |\{ g\in \Lambda_i:~ \exists f\in F, ~fg \notin \Lambda_i \}| }{ |\Lambda_i| } <\delta,$$
\item each $A_i$ is a measurable subset of $X$ with positive measure,
\item $a(\lambda_i)A_i \cap a(\lambda_j ) A_j = \emptyset$ if $i\ne j$, $\lambda_i \in \Lambda_i$ and $\lambda_j \in \Lambda_j$,
\item $a(\lambda )A_i\cap a(\lambda ')A_i =\emptyset$ if $\lambda ,\lambda ' \in \Lambda _i$ and $\lambda\neq \lambda '$,
\item $\mu\left( \cup_{i\in I} \cup_{\lambda \in \Lambda_i} a(\lambda) A_i \right) \ge 1-\eta.$
\end{enumerate}
\end{thm}

\begin{proof}[Proof of Theorem \ref{thm:tiling}]
Let $0<\delta, \eta< \epsilon/2$. Without loss of generality, we assume $e \in F$. Let $\{(\Lambda_i,A_i)\}_{i\in I}$ be as in the theorem above. Define $J$ by $J(x)=\lambda_j$ if there is a $j\in I$ and $\lambda_j \in \Lambda_j$ such that $x \in a(\lambda_j) A_j$. If $x$ is not in $\cup_{i\in I} \cup_{\lambda \in \Lambda_i} a(\lambda) A_i $, then define $J(x)$ arbitrarily. For each $i$, let $\Lambda'_i = \cap_{f\in F} f^{-1} \Lambda_i$. The theorem above implies $|\Lambda'_i| \ge (1-\delta)|\Lambda_i|$. Observe that
$$\{ x\in X:~ J(a(f)x)=fJ(x) ~\forall f\in F\} \supset \cup_{i \in I} \cup_{\lambda \in \Lambda'_i} a(\lambda) A_i.$$
Thus
$$\mu(  \{ x\in X:~ J(a(f)x)=fJ(x) ~\forall f\in F\}) \ge 1 - \eta - \delta \ge 1-\epsilon.$$
\end{proof}

\section{Proof of Theorem \ref{thm:main}}
Assume the hypotheses of Theorem \ref{thm:main}. In particular, we assume that $G/H \cc^b (Y,\nu)$ is a free p.m.p. action of the infinite amenable group $G/H$. For simplicity, if $g\in G$ and $y \in Y$, let $gy$ denote $b(gH)y$. 


Let $F \subset G$ be finite and $\epsilon>0$. Because $G/H$ is amenable, Theorem \ref{thm:tiling} implies there exists a measurable function $J: Y \to G/H$ such that if 
$$Y_0=\{ y\in Y:~ J( fy)=fJ(y) ~\forall f\in F\}$$
then $\nu(Y_0) \ge 1-\epsilon.$ Let $\sigma:G/H \to G$ be a section (i.e., $\sigma(gH) \in gH$ for all $g\in G$). Let $\tilde{J}:Y \to G$ be defined by $\tilde{J} = \sigma J$. 

Recall that $G \cc^a (X,\mu)$ is a p.m.p. action, $Z = \{z \in X^G:~ a(h^{-1}) z(g)=z(gh) \}$ and $G$ acts on $Z$ by $(gz)(f)=z(g^{-1}f)$ for $z \in Z, g,f \in G$. This action is $\CInd_H^G(a|H)$. It preserves the measure $\zeta$ on $Z$ obtained by pulling back the product measure $\mu^{G/H}$ on $X^{G/H}$ under the map $\Phi:Z\to X^{G/H}$, $\Phi(z)(gH)=z(\sigma(gH))$.



For $(z,y) \in Z \times Y$, define $S_y(z) \in X$ by
$$S_y(z) = a(\tilde{J}(y)) z(\tilde{J}(y)).$$

\begin{lem}
The map $(z,y) \in Z \times Y \mapsto S_y(z) \in X$ maps $\zeta \times \nu$ onto $\mu$.
\end{lem}
\begin{proof}
For any $y\in Y$, if $\delta_y$ denotes the Dirac probability measure concentrated on $y$ then it is easy to see that $(z,y) \mapsto S_y(z)$ maps $\zeta \times \delta_y$ onto $\mu$. The lemma follows by integrating over $y$.
\end{proof}

\begin{lem}
For every $(z,y) \in Z\times Y_0$ and $f\in F$, $S_{fy}(fz) = a(f)S_y(z)$.
\end{lem}

\begin{proof}
If $y \in Y_0$ then $J(fy) = fJ(y)$ for all $f\in F$. Therefore, for each $f\in F$ there is some $h\in H$ such that $\tilde{J}(fy)=f\tilde{J}(y)h$. Now
\begin{eqnarray*}
S_{fy}(fz) &=& a(\tilde{J}(fy))(fz)(\tilde{J}(fy))= a(f\tilde{J}(y)h)(fz)(f\tilde{J}(y)h)\\
&=& a(f)a(\tilde{J}(y))a(h)z(\tilde{J}(y)h)= a(f)a(\tilde{J}(y))z(\tilde{J}(y)) = a(f)S_y(z).
\end{eqnarray*}
\end{proof}

Now let $\phi: X \to \N$ be an observable. Define $\psi: Z\times Y \to \N$ by $\psi(z,y) = \phi( S_y(z))$. The lemma above implies that for all $(z,y) \in Z \times Y_0$, $\psi(fz, fy) = \phi (a(f) S_y(z))$ for all $f\in F$. Thus $\psi^F(z,y) = \phi^F(S_y(z))$ for $(z,y) \in Z \times Y_0$. Since $(z,y) \mapsto S_y(z)$ takes the measure $\zeta \times \nu$ to $\mu$ and $\nu(Y_0) \ge 1-\epsilon$, it follows that 
$$\|\psi^F_*(\zeta \times \nu)-\phi^F_*\mu\|_1 < \epsilon.$$
Because $F\subset G$, $\epsilon>0$ and $\phi$ are arbitrary, this implies Theorem \ref{thm:main}.

\section{Consequences of Theorem \ref{thm:main}} \label{sec:conseq}

If $K$ is a group acting on a countable set $T$, then for a measure space $(X,\mu )$ we denote the generalized shift action of $K$ on $(X^T, \mu ^T)$ (given by $(ky)(t)=y(k^{-1}t)$ for $k\in K, y\in X^T$, $t\in T$) by $s_{K,T,X}$.  

\begin{cor}\label{cor:cind}
Let $G$ be a countable group and let $H$ be a normal subgroup of infinite index such that $G/H$ is amenable. Then $a\prec \mbox{CInd}_H^G((a\times \tau _G)|H)$ for every p.m.p. action $a$ of $G$.
\end{cor}

\begin{proof}
Let $(X,\mu)$ be a standard non-atomic probability space. Let $s_{G/H,G/H,X}$ denote the shift of $G/H$ on $X^{G/H}$, which is free.  Let $s_{G,G/H ,X}$ denote the generalized shift of $G$ on $X^{G/H}$.  Then $s_{G,G/H,X}$ is the action of $G$ induced by $s_{G/H,G/H,X}$, i.e., $s_{G,G/H ,X}$ factors through $s_{G/H , G/H , X}$.  By Proposition A.2 of \cite{Ke11} we have that $s_{G,G/H,X} \cong \mbox{CInd}_H^G (s_{H,H/H,X})$. Now $s_{H,H/H,X} = \tau _H$ is just the identity action of $H$ on $X$, and $\tau _H = \tau _G|H$ is the restriction of the identity action of $G$ on $X$ to $H$.

\begin{lem}\label{lem:prod} Let $L$ be a subgroup of the countable group $K$. Let $a,b \in A(L ,X,\mu )$. Then
\[
\mbox{CInd}_L ^K (a) \times \mbox{CInd}_L ^K (b) \cong \mbox{CInd}_L ^K(a\times b)
\]
\end{lem}

\begin{proof}
This is easy to see once we view $\mbox{CInd}_L^K (a)$ as an action on the space $(X^{K/L}, \mu ^{K/L})$ (using the bijection $\Phi :Z\to X^{K/L}$ defined in section \ref{sec:intro}), and similarly view $\mbox{CInd}_L^K (b)$ and $\mbox{CInd}_L^K(a\times b)$ as actions on $(X ^{K/L}, \mu ^{K/L})$ and $((X\times X)^{K/L} , (\mu \times \mu ) ^{K/L} )$ respectively.
\end{proof}

Applying Theorem \ref{thm:main} we now obtain
\begin{align*}
a\prec \mbox{CInd}_H^G (a|H)\times s_{G,G/H ,X} &\cong \mbox{CInd}_H^G (a|H) \times \mbox{CInd}_H^G (\tau _G|H) \cong \mbox{CInd}_H^G ((a\times \tau _G)|H),
\end{align*}
so $a \prec \mbox{CInd}_H^G((a\times \tau_G)|H )$. 
\end{proof}

If in addition to the hypotheses in corollary \ref{cor:cind} we also have $(a\times \tau _G)|H \prec a|H$, then since co-inducing preserves weak containment (A.1 of \cite{Ke11}) it will follow that 
\[
a\prec \mbox{CInd}_H^G ((a\times \tau _G)|H) \prec \mbox{CInd}_H^G (a|H).
\]
Recall that a p.m.p. action $a$ of $G$ on a standard non-atomic probability space is called a \emph{universal} action of $G$ if $b\prec a$ for every p.m.p. action $b$ of $G$. We now have the following.

\begin{cor}\label{cor:a|H}
Let $G$ be a countable group and let $H$ be a normal subgroup of infinite index such that $G/H$ is amenable. Then any one of the following conditions on $a\in A(G,X,\mu )$ implies $a\prec\mbox{CInd}_H^G(a|H)$:
\begin{enumerate}
\item $a$ is ergodic but not strongly ergodic;
\item $a|H$ is ergodic but not strongly ergodic;
\item $a$ is a universal action of $G$;
\item $a|H$ is a universal action of $H$;
\end{enumerate}
In addition, the set of actions $a$ of $G$ for which $a\prec\mbox{CInd}_H^G(a|H)$ is closed under taking products.
\end{cor}

\begin{remark}
The referee points out that condition \emph{2} is in fact strictly stronger than condition \emph{1}. That is, if $G/H$ is amenable then $a|H$ being ergodic but not strongly ergodic implies that $a$ itself is not strongly ergodic. This is a special case of \cite{Io10} lemma 2.3.
\end{remark}

\begin{proof}[Proof of \ref{cor:a|H}]
\emph{3} and \emph{4} are immediate from corollary \ref{cor:cind}, and \emph{1} and \emph{2} follow from \ref{cor:cind} along with Theorem 3 of \cite{AW11} where they show that $a\times \tau _G \prec a$ holds for ergodic $a$ that are not strongly ergodic. The last statement follows from \ref{lem:prod} since if $a\prec\mbox{CInd}_H^G(a|H)$ and $b\prec\mbox{CInd}_H^G(b|H)$ then $a\times b \prec \mbox{CInd}_H^G(a|H)\times \mbox{CInd}_H^G(b|H)\cong \mbox{CInd}_H^G((a\times b)|H)$.
\end{proof}

We can now prove Theorem \ref{thm:univ}
\begin{proof}[Proof of \ref{thm:univ}]
Suppose $b$ is a universal action of $H$. Let $a$ be a universal action of $G$. It suffices to show that $a\prec\mbox{CInd}_H^G(b)$. We have $a|H\prec b$ by universality of $b$, and so by \emph{3} of Corollary \ref{cor:a|H} we have that $a\prec \mbox{CInd}_H^G(a|H ) \prec \mbox{CInd}_H^G(b)$.
\end{proof}


\begin{remark}The assumption that $G/H$ is amenable is in some cases necessary in order for $\mbox{CInd}_H^G$ to preserve universality. That is, there are examples of groups $H\leq G$ with $H$ infinite index in $G$ such that $G/H$ is \emph{not} amenable, and such that $a\mapsto \mbox{CInd}_H^G (a)$ does not map universal actions to universal actions. For example, if $H$ is any subgroup of infinite index in a group $G$ with property (T) (e.g., if $G=H\times K$ where both $H$ and $K$ are countably infinite with property (T)) then $\mbox{CInd}_H^G(b)$ is weak mixing for every $b\in A(H,X,\mu )$ (see \cite{Io08} lemma 2.2 (ii)), hence is \emph{never} universal. Another example is when $H$ is amenable and $G/H$ is non-amenable (e.g., if $G= H\times K$ where $H$ is any amenable group and $K$ is any non-amenable group).  This implies that $G$ is non-amenable. If $s=s_{H,H,X}$ is the shift of $H$ on $(X^H , \mu ^H)$ then $s$ is universal for $H$ since $H$ is amenable, but $\mbox{CInd}_H^G (s) \cong s_{G,G,X}$ is not universal since $G$ is non-amenable. 
\end{remark}

\begin{remark}
In case $H$ is finite index in $G$ then we actually have the following form of Theorem \ref{thm:main}. We do not assume that $H$ is normal in $G$. Let $b$ denote the action of $G$ on $G/H$, where we view $G/H$ as equipped with normalized counting measure $\nu$.  Then for any p.m.p. action $a$ of $G$ on a standard non-atomic probability space $(X,\mu )$, $a$ is a \emph{factor} of $\mbox{CInd}_H^G (a|H)\times b$.  One way to see this is to use the isomorphism $\mbox{CInd}_H^G(a|H)\cong a^{G/H}\circledast s_{G,G/H ,X}$ given by proposition A.3 of \cite{Ke11}. Here $a^{G/H}\circledast s_{G,G/H ,X}$ is the p.m.p. action of $G$ on $(X^{G/H},\mu ^{G/H})$ given by $a^{G/H}\circledast s_{G,G/H ,X}(g) = a^{G/H}(g) \circ s_{G,G/H,X}(g)$ (note that the transformations $a^{G/H}(g)$ and $s_{G,G/H , X}(g)$ commute for all $g\in G$). Then $(a^{G/H}\circledast s_{G,G/H ,X})\times b$ is an action on the space $(X^{G/H} \times G/H , \mu ^{G/H}\times \nu)$, and the map $(f,gH)\mapsto f(gH) \in X$ factors this action onto $a$.
\end{remark}

\section{Gaussian actions}\label{sec:gactions}

A (real) \emph{positive definition function} $\varphi : I\times I\ra \R$ on a countable set $I$ is a real-valued function satisfying $\varphi (i,j) = \varphi (j,i)$ and $\sum _{i,j\in F}a_ia_j\varphi (i,j) \geq 0$ for all finite $F\subseteq I$ and reals $a_i$, $i\in F$.

\begin{thm}\label{thm:gmeas}
If $\varphi :I\times I \ra \R$ is a real valued positive definite function on a countable set $I$, then there is a unique Borel probability measure $\mu _\varphi$ on $\R ^I$ such that the projection functions $p_i : \R ^I \ra \R$, $p_i (x) = x(i)$ ($i\in I$), are centered jointly Gaussian random variables with covariance matrix $\varphi$. That is, $\mu _\varphi$ is uniquely determined by the two properties
\begin{enumerate}
\item Every finite linear combination of the projection functions $\{ p_i \} _{i\in I}$ is a centered Gaussian random variable on $(\R ^I , \mu _\varphi )$;
\item $\E (p_i p_j) = \varphi (i,j)$ for all $i,j\in I$.
\end{enumerate}
\end{thm}
For a finite $F\subseteq I$, let $p _F : \R ^I\ra \R ^F$ be the projection $p _F (x) = x|F$. Then $\mu _\varphi$ can also be characterized as the unique Borel probability measure on $\R ^I$ such that for each finite $F\subseteq I$ the measure $(p _F)_*\mu _\varphi$ on $\R ^F$ has characteristic function
\[
\widetilde{(p_F) _*\mu _\varphi }(u)= e^{-\tfrac{1}{2}\sum _{i,j\in F}u_iu_j\varphi (i,j)} .
\]
We call $\mu _\varphi$ the \emph{Gaussian measure associated to $\varphi$} and $(\R ^I , \mu _\varphi )$ a \emph{Gaussian probability space}.  A discussion of this can be found in \cite{Ke10} Appendix C and the references therein.
\\
\indent Let $G$ be a countable group acting on $I$ and suppose that the positive definite function $\varphi : I\times I \ra \R$ is invariant for the action of $G$ on $I$, i.e., $\varphi (g\cdot i, g\cdot j) = \varphi (i,j)$ for all $g\in G$, $i,j\in I$.  Let $s_\varphi$ denote the shift action of $G$ on $(\R ^I , \mu _\varphi )$
\[
(s_\varphi (g)x)(i) = x(g^{-1}\cdot i).
\]
Then invariance of $\varphi$ implies that $\mu _\varphi$ is an invariant measure for this action. We call $s_\varphi$ the \emph{Gaussian shift associated to $\varphi$}. 
\\
\indent Let $\pi$ be an orthogonal representation of $G$ on a separable real Hilbert space $\cH _\pi$, and let $T\subseteq \cH _\pi$ be a countable $\pi$-invariant set whose linear span is dense in $\cH _\pi$. Then $G$ acts on $T$ via $\pi$, and we let $\varphi _T: T\times T \ra \R$ be the $G$-invariant positive definite function given by $\varphi _T(t_1, t_2) = \langle t_1, t_2 \rangle$. We let $s_\pi = s_{\pi , T}$ be the corresponding Gaussian shift and call it the \emph{Gaussian shift action associated to $\pi$}.  It follows from proposition \ref{prop:gact} below that up to isomorphism this action does not depend on the choice of $T\subseteq \cH _\pi$.  For now, it is clear that an isomorphism $\theta$ of two representations $\pi _1$ and $\pi _2$ induces an isomorphism of the actions $s_{\pi _1, T}$ with $s_{\pi _2, \theta (T)}$.
\\
\indent By the GNS construction, every invariant real positive definite function $\varphi$ on a countable $G$-set may be viewed as coming from an orthogonal representation in this way.
\\
\indent There is another way of obtaining an action on a Gaussian probability space from an orthogonal representation of $G$. Consider the product space $(\R ^\N , \mu ^\N )$, where $\mu$ is the $N(0,1)$ normalized, centered Gaussian measure on $\R$ with density $\frac{1}{\sqrt{2\pi} }e^{-x^2 /2}$.  Let $p_n : \R ^\N \rightarrow\R$, $n\in \N$, be the projection functions $p_n(x)= x(n)$. The closed linear span $\langle p_n \rangle _{n\in \N} \subseteq L^2 (\R ^\N , \mu ^\N ,\R )$ has countable infinite dimension. Let $\cH = \langle p_n \rangle _{n\in\N } \subseteq L^2 (\R ^\N , \mu ^\N , \R )$ and let $\pi$ be a representation of $G$ on $\cH$.  Let $a(\pi )$ be the action of $G$ on $(\R ^\N , \mu ^\N )$ given by
\[
(a(\pi ) (g)x)(n) = \pi (g^{-1})(p_n)(x).
\]
This preserves the measure $\mu ^\N$ by the characterization of $\mu ^\N$ given in \ref{thm:gmeas} since $\mu ^\N = \mu _\varphi$, where $\varphi :\N \times \N \ra \R$ is the positive definite function given by $\varphi (n,n) =1$ and $\varphi (n,m)= 0$ for $n\neq m$.
\\
\indent It follows from the discussion in \cite{Ke10} Appendix E that if $\pi _1$ and $\pi _2$ are isomorphic, then $a(\pi _1)\cong a(\pi _2)$.  So if $\pi$ is now an arbitrary orthogonal representation of $G$ on an infinite dimensional separable real Hilbert space $\cH _{\pi}$, then by choosing an isomorphism $\theta$ of $\cH _\pi$ with $\cH =\langle p_n \rangle _{n\in \N}$ we obtain an isomorphic copy $\theta \cdot \pi$ of $\pi$, on $\cH$, and the corresponding action $a(\theta \cdot \pi )$ is, up to isomorphism, independent of $\theta$.
\\
\indent The construction of the actions $a(\pi )$ also works for representations on a finite dimensional Hilbert space, replacing $\N$ above with $N = \mbox{dim}(\mc{H}_\pi )$.  The following proposition also holds in the finite dimensional setting.

\begin{prop}\label{prop:gact}
Let $\pi$ be an orthogonal representation of $G$ on $\cH = \langle p_n \rangle _{n\in \N} \subseteq L^2 (\R ^\N , \mu ^\N , \R )$, let $T\subseteq \cH$ be a countable $\pi$-invariant set of functions in $\mc{H}$ whose linear span is dense in $\cH$, and let $s_{\pi , T}$ be the corresponding Gaussian shift on $(\R ^T , \mu _{\varphi _T})$.  Then the map $\Phi :(\R ^\N , \mu ^\N ) \ra (\R ^T , \mu _{\varphi _T} )$ given by
\[
\Phi (x)(t) = t(x)
\]
is an isomorphism of $a(\pi )$ with $s_{\pi ,T}$.  In particular, up to isomorphism, the action $s_{\pi ,T}$ does not depend on the choice of $T$.
\end{prop}

\begin{proof}
Note that up to a $\mu ^\N$-null set, $\Phi$ does not depend on the choice of representatives for the elements of $T$ (viewing each $t\in T$ as an equivalence class of functions in $L^2(\R ^\N , \mu ^\N ,\R )$). This follows from $T$ being countable. So $\Phi$ is well defined.
\\
\indent To see that $\Phi _*(\mu ^\N ) = \mu _{\varphi _T}$ we use \ref{thm:gmeas}. First, we show that if $f = \sum _{i=1}^k a_i p_{t_i}$ then $f$ has a centered Gaussian distribution with respect to $\Phi _*(\mu ^\N )$.  This is clear since $f_*\Phi _* (\mu ^\N ) = (f\circ \Phi )_*(\mu ^\N )$, and $f\circ \Phi =\sum _{i=1}^k a_i t_i$ has centered Gaussian distribution with respect to $\mu ^\N$ by virtue of being in $\cH$.
\\
\indent Second, we show that the covariance matrix of the projections $\{ p_t \} _{t\in T}$ with respect $\Phi _*\mu ^\N$ is equal to $\varphi _T$. We have
\begin{align*}
\int p_{t_1}(x) p_{t_2}(x) \, d(\Phi _*\mu ^\N ) &= \int \Phi (x)(t_1)\Phi (x)(t_2) \, d(\mu ^\N ) \\
&= \int t_1 t_2 \, d(\mu ^\N ) = \langle t_1 ,t_2 \rangle = \varphi (t_1,t_2) .
\end{align*}
Next, we show that $\Phi$ takes the action $a_\pi $ to the action $s_{\pi ,T}$. We have, for $\mu ^\N$-a.e. $x$,
\begin{align*}
\Phi (a(\pi )&(g)x)(t) = t(a(\pi )(g)x) = \textstyle\sum _n \langle t, p_n \rangle p_n (a(\pi )(g)x)  = \textstyle\sum _n \langle t, p_n \rangle \pi (g ^{-1})(p_n)(x)\\
& = \pi (g ^{-1})(\textstyle\sum _n \langle t, p_n \rangle p_n)(x) = \pi (g ^{-1})(t)(x) = \Phi (x)(\pi (g^{-1})(t)) = s_{\pi ,T} (g)(\Phi (x))(t).
\end{align*}
It remains to show that $\Phi$ is 1-1 on a $\mu ^\N$-measure 1 set. Since the closed linear span of $\{ t \} _{t\in T}$ in $\mc{H}$ contains each $p_i$, it follows that the $\sigma$-algebra generated by $\{ t \} _{t\in T}$ is the Borel $\sigma$-algebra modulo $\mu ^\N$-null sets, so there is a $\mu ^\N$-conull set $B$ such that $\{ t|B\} _{t\in T}$ generates the Borel $\sigma$-algebra of $B$ and thus $\{ t|B \}$-separates points. It follows that $\Phi$ is 1-1 on $B$.
\end{proof}


\section{Induced representations and the proof of Theorem \ref{thm:gaussian}}

We begin by briefly recalling the induced representation construction. Let $H$ be a subgroup of the countable group $G$, and let $\sigma :G/H \ra G$ be a selector for the left cosets of $H$ in $G$ with $\sigma (H)=e$.  Let $\rho : G\times G/H \ra H$ be defined by $\rho (g , kH) = \sigma (gkH)^{-1}g\sigma (kH) \in H$. Then $\rho$ is a cocycle for the action of $G$ on $G/H$, i.e., $\rho (g_0g_1, kH) = \rho (g_0, g_1kH)\rho (g_1,kH)$.  (Note that this is the same as the cocycle $\rho$ defined in the proof of Lemma \ref{lem:prod}.)

Let $\pi$ be an orthogonal representation of $H$ on the real Hilbert space $\mc{K}$. For each $gH\in G/H$ let $\mc{K}_{gH} = \mc{K}\times \{ gH \} = \{ (\xi ,gH)\csuchthat \xi \in \mc{K} \}$ be a Hilbert space which is a copy of $\mc{K}$. Then the induced representation $\mbox{Ind}_H^G(\pi )$ of $\pi$ is the representation of $G$ on $\bigoplus _{g\in G/H}\mc{K}$, which we identify with the set of formal sums $\mc{K}' = \{ \sum _{gH\in G/H} (\xi _{gH}, gH) \in \sum _{gH\in G/H}\mc{K}_{gH} \csuchthat \sum _{gH\in G/H} ||\xi _{gH} ||_{\mc{K}}^2 <\infty \}$, that is given by
\[
g_0 \cdot (\xi _{gH}, gH) = (\rho (g_0 ,gH)\cdot \xi _{gH} , \, g_0gH) \in \mc{K}_{g_0gH}
\]
for $(\xi _{gH}, gH) \in \mc{K}_{gH}$, and extending linearly.

\begin{lem}\label{lem:commute}
Let $H$ be a subgroup of the countable group $G$. Then
\begin{enumerate}
\item $a(\pi |H) \cong a(\pi )|H$ for all orthogonal representations $\pi$ of $G$.
\item $\mbox{CInd}_H^G(a(\pi )) \cong a(\mbox{Ind}_H^G(\pi ))$ for all orthogonal representations $\pi$ of $H$.
\end{enumerate}
\end{lem}

\begin{proof}
The first statement is clear. For the second, let $T\subseteq \mc{K}$ be a total, countable subset of $\mc{K}$ that is invariant under $\pi$. 
Then $T\times G/H \subseteq \mc{K}'$ is a total, countable subset of $\mc{K}'$ that is invariant under $\mbox{Ind}_H^G(\pi )$.  
Let $\varphi : (T\times {G/H}) \times (T\times G/H) \ra \R$ be the inner product determined by
\[
\varphi ((t_1, g_1H),(t_2,g_2H)) = \langle (t_1,g_1H) , (t_2, g_2H)\rangle _{\mc{K}'} = 
\begin{cases}
\langle t_1, t_2 \rangle _{\mc{K}} & \mbox{ if }g_1H=g_2H \\
0 & \mbox{ if }g_1H\neq g_2H .
\end{cases}
\]
Then the Gaussian shift action corresponding to $\mbox{Ind}_H^G(\pi )$ is the action $b$ of $G$ on $(\R ^{T\times G/H} ,\mu _\varphi )$ given by 
\[
(b(g)\cdot x)((t,kH)) = x(g^{-1}\cdot (t,kH)) = x ( (\rho (g^{-1}, kH) \cdot t, g^{-1}kH )) .
\]
On the other hand, the Gaussian shift action corresponding to $\pi$ is the action $s_{\pi} \cong a(\pi)$ of $H$ on $(\R ^T , \mu _{\varphi _T})$ given by $(s_{\pi} (h)\cdot w)(t) = w(h^{-1}\cdot t)$, and where $\varphi _T : T\times T \ra \R$ is just the inner product $\varphi _T(t_1,t_2) = \langle t_1, t_2 \rangle _{\mc{K}}$.  The co-induced action $\mbox{CInd}_H^G (s_{\pi} )$ is isomorphic to the action $c$ of $G$ on $((\R ^T)^{G/H} , \mu _{\varphi}^{G/H} )$ given by $(c(g)\cdot y)(kH) = s_{\pi} (\rho (g^{-1} , kH)^{-1}) \cdot y(g^{-1}kH)$. Evaluating this at $t\in T$ gives
\[
(c(g)\cdot y)(kH)(t) = (s_{\pi} (\rho (g^{-1} , kH)^{-1})\cdot y(g^{-1}kH))(t) = y(g^{-1}kH)(\rho (g^{-1} , kH) \cdot t) .
\]
It follows that the bijection $\Psi : \R ^{T\times G/H}  \ra (\R ^T)^{G/H}$ given by $\Psi (x)(kH)(t) = x((t,kH))$ takes the action $b$ to the action $c$, and also takes the measure $\mu _{\varphi}$ to $\mu _{\varphi _T}^{G/H}$. So $b\cong c$ as was to be shown.
\end{proof}

If $\pi _1$ and $\pi _2$ are orthogonal representations of $G$ on $\mathcal{H}_1$ and $\mathcal{H}_2$ respectively, then we say $\pi_1$ is {\em weakly contained in $\pi_2$ in the sense of Zimmer} \cite{Zi84} and write $\pi _1 \prec _Z \pi _2$ if for all $v_1,\dots ,v_n \in \mc{H}_1$, $\epsilon >0$, and $F\subseteq G$ finite, there are $w_1,\dots ,w_n\in \mc{H}_2$ such that $|\langle \pi _1 (g)(v_i),v_j\rangle - \langle \pi _2 (g)(w_i),w_j\rangle | <\epsilon$ for all $g\in F$, $i,j\leq n$.

\begin{lem}\label{lem:zim}
$\pi _1 \prec _Z \pi _2 \Rightarrow a(\pi _1)\prec a(\pi _2)$.
\end{lem}

\begin{proof}
This is the remark after Theorem 11.1 of \cite{Ke10}.
\end{proof}

\begin{lem}\label{lem:ind}
Let $G$ be a countable group with a subgroup $H<G$. Suppose the action of $G$ on $G/H$ is amenable. Then $\pi \prec _Z \mbox{Ind}_H^G(\pi |H)$ for every orthogonal representation $\pi$ of $G$.
\end{lem}

\begin{proof}
It is well known that the action of $G$ on $G/H$ being amenable is equivalent to the existence of a sequence $u_n$, $n\in \N$, of unit vectors in $l^2(G/H , \R )$ that are asymptotically invariant for the quasi-regular representation $\lambda _{G/H}$ of $G$ (given by $\lambda _{G/H}(g_0)(\delta _{g_1H})=\delta _{g_0g_1H}$ where $\delta _{gH} \in l^2(G/H ,\R )$ is the indicator of $\{ gH \}$).  This means that for every $g\in G$, $\langle \lambda _{G/H}(g)(u_n),u_n \rangle \ra 1$ as $n\ra \infty$.  

Let $\mc{K}$ be the Hilbert space of $\pi$. The representation $\mbox{Ind}_H^G(\pi |H)$ is isomorphic to $\pi \otimes \lambda _{G/H}$ (this is E.2.6 of \cite{BdlHV08}); an isomorphism is given by (extending linearly) the map that sends $(\xi ,gH) \in \mc{K}_{gH}$ to $\pi (\sigma (gH))(\xi ) \otimes \delta _{gH} \in \mc{K}\otimes l^2(G/H,\R )$.  Given now $v_1,\dots ,v_n \in \mc{K}$, $\epsilon >0$, and $F\subseteq G$ finite, we have that for all $N$ sufficiently large
\begin{align*}
&|\langle \pi (g)(v_i), v_j \rangle - \langle (\pi \otimes \lambda _{G/H})(g) (v_i \otimes u_N), v_j \otimes u_N\rangle |\\
 &=\left|\langle \pi (g)(v_i),v_j \rangle \left(1-\langle \lambda _{G/H}(g)(u_N), u_N\rangle \right) \right|  <\epsilon
\end{align*}
for each $g\in F$, $i,j\leq n$. So taking $w_i = v_i\otimes u_N$ for $N$ sufficiently large shows that $\pi \prec _Z \pi \otimes \lambda _{G/H}\cong \mbox{Ind}_H^G(\pi |H )$.
\end{proof}


\begin{proof}[Proof of Theorem \ref{thm:gaussian}]
Let $\pi$ be an orthogonal representation of $G$ such that $a\cong a(\pi )$.  Then $\pi \prec _Z \mbox{Ind}_H^G(\pi |H)$ by Lemma \ref{lem:ind}. Applying Lemma \ref{lem:zim} and then Lemma \ref{lem:commute} we obtain
\[
a(\pi )\prec a(\mbox{Ind}_H^G(\pi |H)) \cong \mbox{CInd}_H^G (a(\pi |H)) \cong \mbox{CInd}_H^G(a(\pi )|H). \qedhere
\]
\end{proof}

\begin{remark}
An alternative proof of Theorem \ref{thm:gaussian} can be given that uses probability theory. For a Gaussian shift action $s_\varphi$ on $(Y,\nu ) = (\R ^T ,\mu _\varphi )$ one may identify $\mbox{CInd}_H^G(s_{\varphi}|H)$ with the isomorphic action $b= s_\varphi ^{G/H}\circledast s_{G,G/H , Y}$ (see A.3 of \cite{Ke11}) on $(Y^{G/H}, \nu ^{G/H})$. Using an appropriate F\o lner sequence $\{ F_n \}$ for the action of $G$ on $G/H$ one defines the maps $p_n: Y^{G/H}\ra Y$, $p_n(w)=|F_n|^{-1/2}\sum_{x\in F_n}w(x)$, each factoring the action $s_\varphi ^{G/H}$ onto $s_\varphi$.  Then using arguments as in \cite{KT08} it can be shown that for cylinder sets $A\subseteq Y$, the sequence $p_n^{-1}(A)$, $n\in \N$, is asymptotically invariant for $s_{G,G/H,Y}$, from which it follows that $s_\varphi \prec b$.
\end{remark}

\end{document}